\documentclass[11pt]{amsart}
\usepackage{pdfsync}
\usepackage{wasysym}
\usepackage{enumerate}
\usepackage{graphicx}
\usepackage{latexsym}
\usepackage{amsfonts}
\usepackage[utf8]{inputenc}
\usepackage{amsthm}
\usepackage{amsmath}
\usepackage{amssymb}
\usepackage[
            unicode,
            breaklinks=true,
            colorlinks=true]{hyperref}
\usepackage[active]{srcltx}
\usepackage{amscd}
\usepackage[all]{xy} 
\usepackage{mathrsfs}
\usepackage{stmaryrd} 
\usepackage{amsopn} 
\usepackage{eepic}
\usepackage{epic}
\usepackage{eucal}
\usepackage{epsfig}
\usepackage{psfrag}

\newtheorem{theorem}{Theorem}[section]
\newtheorem{prop}[theorem]{Proposition}

\newtheorem{lemma}[theorem]{Lemma}

\newcommand{\R}{{\mathbb R}}

\newcommand{\Z}{{\mathbb Z}}

\newcommand{\N}{{\mathbb N}}

\renewcommand{\le}{\leq}

\newcounter{todo}
\newcommand{\op}[1]{\!\!\mathop{\rm ~#1}\nolimits}

\renewcommand{\geq}{\geqslant}
\renewcommand{\leq}{\leqslant}

\def\Hil{\mathcal{H}}

\newenvironment{remark}{\refstepcounter{theorem}\par\medskip\noindent{\bf
Remark~\thetheorem.}}{\unskip\nobreak\hfill\hbox{ $\oslash$}\par\bigskip}

\begin{document}

\title[Euler\--MacLaurin formulas]{Euler\--MacLaurin formulas\\
via differential operators}
\author{Yohann Le Floch  \,\,\,\,\,\,\,\,\,\,\, \,\,\,\,\, \'Alvaro Pelayo}

\maketitle

\begin{abstract}
Recently there has been a renewed interest in asymptotic
Euler\--MacLaurin formulas, partly due to applications to spectral theory of differential operators.
Using elementary means, we recover such formulas for compactly supported
smooth functions $f$ on intervals, polygons, and $3$\--dimensional polytopes 
$\Delta$, where the coefficients in the asymptotic
expansion are sums of differential
operators involving only derivatives of $f$ in directions normal to
the faces of $\Delta$. Our formulas  apply to wedges of 
any dimension. This paper builds on, and is motivated by, works
of Guillemin, Sternberg, and others, in the past ten years.
\end{abstract}

\section{Introduction} \label{sec:intro}

\subsection{The context}

Let $\Delta$ be a polytope in $\R^n$. Euler\--MacLaurin formulas are expressions which may be used to 
approximate Riemann sums of smooth functions $f \colon \mathbb{R}^n \to \R$
such as
\begin{equation} \frac{1}{N^n} \sum_{k \in \mathbb{Z}^n \cap N\Delta} f \Big( \frac{k}{N}\Big), \label{eq:sum}\end{equation}
in terms
of integrals involving $f$. Such formulas may be traced back to the work of L. Euler and C. MacLaurin in the first half
of the XIX century. Partly due to its connections to other areas of mathematics, this topic  has attracted great interest recently,
see for instance Berline\--Vergne~\cite{BeVe2007}, Brion\--Vergne~\cite{BrVe1997}, Cappell\--Shaneson~\cite{CS},
Guillemin\--Sternberg~\cite{GuSt2007}, Guillemin\--Sternberg\--Weitsman~\cite{GSW}, Guillemin\--Stroock~\cite{GS}, Guillemin\--Wang~\cite{GuWa08}, Karshon\--Sternberg\--Weitsman~\cite{KSW}, Shaneson~\cite{Sh94}, Tate \cite{Ta2010}, and Vergne~\cite{Ve13}. Let us also mention that in \cite[Section~5, Theorem 5.9]{ChVN08} Charles and  V\~u Ng\d oc obtained an asymptotic expansion for a sum over integral points of a convex polytope as a consequence of a trace formula, and computed explicitly the first term of this expansion.

Our  interest in Euler\--MacLaurin formulas originates mainly in the applications they have in spectral theory of differential operators, see for instance the recent articles of Burns\--Guillemin\--Wang~\cite{BuGuWa} and Zelditch~\cite{Ze2009}. Euler\--MacLaurin formulas are also 
connected to major problems in number theory, see eg. Lagarias~\cite{La2013}. At least in dimension one, these formulas have many other applications in other areas, for instance in numerical analysis and computational mathematics, see e.g. Berrut~\cite{Ber} and Tausch~\cite{Tau10}, in actuarial sciences, for instance simple approximations in terms of annuities with annual payments are available based on Euler\--MacLaurin formulas (see \cite{Mac} and the references therein), and in the study of the Casimir effect problem, see for instance \cite{Cas,Dow}. In \cite{CosGar}, Costin and Garoufalidis describe applications of Euler\--MacLaurin formulas to linear difference equations with a small parameter, and to the study of knotted 3-dimensional objects.

Recently, Tate~\cite{Ta2010} obtained concrete Euler\--MacLaurin formulas to approximate Riemann sums of the form (\ref{eq:sum}) for smooth functions on lattice polytopes in any dimension, building primarily on the work of Berline and Vergne~\cite{BeVe2007}. As a consequence of this general result and its rather involved proof, Tate obtained a fully explicit formula for Riemann sums over Delzant polygons (Figure~\ref{fig:polytopes}), that is Delzant polytopes in dimension 2. This case is already of interest because such Delzant polygons arise as images of the moment maps of 4-dimensional toric manifolds.

Let us briefly explain how toric manifolds enter the picture. Let $(M,\omega)$ be a compact connected symplectic toric manifold, $\mu:M \to \R^n$ the underlying moment map, and $\Delta = \mu(M)$ the associated Delzant polytope. Assume that there exists a prequantum line bundle $L \to M$, that is a holomorphic, Hermitian line bundle which Chern connection has curvature $-i\omega$, and consider the space $\Hil_k = H^0(M,L^{\otimes k})$ of holomorphic sections of $L^{\otimes k}$; $\Hil_k$ is the Hilbert space of geometric quantization. A standard result states that the dimension of $\Hil_k$ is equal to the number ${\string #} \left( \Z^n \cap k\Delta/(2\pi)\right)$ of integer points in the polytope $k\Delta/(2\pi)$ (see \cite{Ham} for a nice exposition in the case $k=1$\footnote{And with a different convention which explains some differences by factors $2\pi$.}), that is the so-called Ehrhart polynomial \cite{Ehr} of this polytope. An analogous result for toric varieties was one of the motivations behind works by Khovanski{\u\i}~\cite{Khov77,Khov78}, followed by Khovanski{\u\i} and Pukhlikov~\cite{KhovPuk92a,KhovPuk92b} (see also Kantor and Khovanski{\u\i}~\cite{KanKhov93}), providing formulas for sums of values of $f$ on lattice points of a lattice polytope, where $f$ is of the form polynomial times exponential.\\
More generally, let $(T_{1,k},\ldots,T_{n,k})$ be commuting self-adjoint Berezin-Toeplitz operators (each $T_{i,k}$ acts on the space $\Hil_k$ defined above; for more details, see for instance \cite{Schli10} and references therein) whose principal symbols are the components of the moment map $\mu:M \to \R^n$. A result recently obtained by Charles, Pelayo and V\~u Ng\d oc~\cite[Theorem $1.1
1$]{ChPeVN2013} states that the joint spectrum $\mathcal{JS}$ (consisting of joint eigenvalues) of $T_{1,k},\ldots,T_{n,k}$ coincides, up to $\mathcal{O}(k^{-\infty})$, with the set
\[ g\left(\Delta \cap \left(v + \frac{2\pi}{k} \Z^n \right); k\right) \] 
where $v$ is any vertex of the Delzant polytope $\Delta = \mu(M)$ and the sequence of functions $g(.;k):\R^n \to \R^n$ admits a ${\rm C}^{\infty}$-asymptotic expansion of the form $g(.,k) = \mathrm{Id} + k^{-1} g_1 + k^{-2} g_2 + \ldots$, where each $g_j: \R^n \to \R^n$ is smooth. Thus, given any smooth function $F: \R^n \to \R$, the trace of the operator $F(T_{1,k}, \ldots, T_{n,k})$ is equal to
\[ \sum_{\lambda \in \mathcal{JS}} F(\lambda) = \sum_{x \in \Z^n \cap k\Delta/(2\pi)} \varphi\left(\frac{x}{k}\right), \]
where $\varphi: \R^n \to \R$ is defined by the formula $\varphi(x) = F(g(v + 2\pi x))$ (the particular case $F=1$ corresponds to the computation of the dimension of $\Hil_k$, see discussion above). The right-hand side of this equality is of the form (\ref{eq:sum}), up to a multiplicative factor $k^n$, and can thus be evaluated thanks to Euler\--MacLaurin formulas.

\subsection{Aim and structure of the paper}

The goal of this paper is double. Firstly, we recover some of Tate's results by elementary methods--different from those used in~\cite{Ta2010}--starting from a result due to Guillemin and Sternberg~\cite{GuSt2007}, which was itself derived by elementary means. More precisely, we obtain explicit Euler\--MacLaurin formulas to approximate (\ref{eq:sum}) for compactly supported functions $f$ on wedges in $\R^n$, intervals, polygons, and polytopes $\Delta$ in $\R^3$, where the coefficients in the asymptotic expansion are sums of differential operators involving only derivatives of $f$ in directions normal to the faces of $\Delta$. Secondly, we provide explicit examples in dimension 2; the first one is fully explicit, in the sense that the asymptotic expansion of (\ref{eq:sum}) is finite and that we provide all of its coefficients, while for the second one we provide some numerical evidence of the correctness of this asymptotic expansion.

The only reason to restrict ourselves to intervals, polygons and $3$\--polytopes, is because of technical difficulties due to our method of proof. We hope to be able to extend our techniques to higher dimensions in future works. However, we believe that the methods and results presented here are themselves already interesting, understandable in the sense that they only involve elementary computations, and may serve as a comprehensive introduction to the more general topic of Euler\--MacLaurin formulas over lattice polytopes in any dimension as treated in~\cite{Ta2010}. Furthermore, despite the simplicity of the computations involved in the cases that we treat, we have not seen them explicitly written anywhere, which constitutes another reason to think that this proof is worth presenting on its own.

The structure of the paper is as follows: in Section~\ref{mainM} we will state our main result, and the remaining sections of the
paper are devoted to its proof and further refinements. In the case of polygons $\Delta$, our approach leads to formulas in which the coefficients in the asymptotic expansion can be explicitly calculated, and we do so in Section~\ref{n2case}. In the Appendix, we recall some of Tate's notation and results and provide a comparison with ours.

\section{Main Result} \label{mainM}

Let $n\geq 1$ and
let $\mathbb{Z}^n$ be the integer lattice in $\mathbb{R}^n$. Let
$(\mathbb{Z}^n)^*$ and $(\mathbb{R}^n)^*$ be the corresponding
dual spaces. Let $\langle \cdot,\, \cdot \rangle$ denote the pairing
of $(\mathbb{R}^n)^*$ with $\mathbb{R}^n$. The subset $W \subset \mathbb{R}^n$ given by the inequalities 
\begin{eqnarray} \label{for:1}
\langle u_i, \, x \rangle \le c_i,\,\,\,\, i \in \{1, \ldots, m\}
\end{eqnarray}
is called an \emph{integer $m$\--wedge} if
for every $i \in \{1,\ldots,m\}$ the constant  $c_i$ is an integer
and the vector $u_i$ is a primitive lattice vector in $(\mathbb{R}^n)^*$.  Let $U$ be 
the subspace of $(\R^n)^*$ spanned by the $u_i$'s. We say
that $W$ is \emph{regular} if $\{u_1, \ldots,u_m\}$
is a lattice basis of the lattice $U \cap (\mathbb{Z}^n)^*$.

\begin{figure}
\begin{center}
\epsfig{file=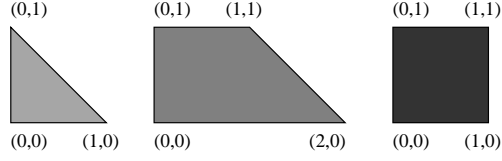}
\end{center}
\caption{Delzant polygons.}
\label{fig:polytopes}
\end{figure}

When $m=n$, define the diffeomorphism 
$\varphi \in {\rm C}^{\infty}(\mathbb{R}^n,\mathbb{R}^n)$, such that $\varphi(x) = y$ has 
coordinates $y_i := \langle u_i,x \rangle - c_i$, $ i \in \{1,\ldots,n\}$,  in the standard orthonormal basis of $\R^n$.
Let $\mathcal{H}_i$ be the facet  of $W$ defined by
$$\mathcal{H}_i := \left\{ x \in W \,\, | \,\, \langle u_i,x \rangle = c_i \right\}.$$
For $\alpha=(\alpha_i)_{1 \leq i \leq n} \in \N^n$, set $$F = \bigcap_{i, \alpha_i > 0} \mathcal{H}_i$$ and let the constant $K_{\alpha}(W)$ be defined as the Jacobian of the diffeomorphism
$\varphi_{|F}: F \to \bigcap_{i, \alpha_i > 0} \{y_i=0\}$,
and write
\[ \varint^{\star}_{F} := K_{\alpha}(W) \varint_{F}. \] 
Let $(v_i)_{1 \leq i \leq n}$ be the dual basis of $(u_i)_{1 \leq i \leq n}$.
The constant $K_{\alpha}(W)$ is the inverse of the volume of the parallelotope generated by the vectors $v_i$ for $i$ such that $\alpha_i = 0$, that is, the primitive outwards vectors defining the face $F$ (see Figure \ref{fig:kalpha}).

\begin{figure}
\begin{center}
\epsfig{file=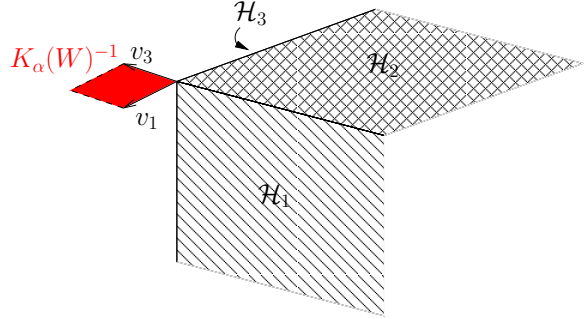,scale=0.6}
\end{center}
\caption{Computation of $K_{\alpha}(W)$; here $\alpha = (0,1,0)$.}
\label{fig:kalpha}
\end{figure}

We introduce the following notation:
\begin{itemize}
\item  $\alpha ! = \alpha_1 ! \ldots \alpha_n !$;
\item  $r(\alpha) \in \mathbb{N}^n$ is given by $r(\alpha)_i = 1$ if $\alpha_i > 0$, $r(\alpha)_i = 0$ if $\alpha_i = 0$;
\item if $u_1,\ldots,u_n \in \mathbb{R}^n$, $u_{\alpha}$ is the $|\alpha|$-tuple of vectors
\begin{equation*} u_{\alpha} = (\underbrace{u_1,\ldots,u_1}_{\alpha_1 \mathrm{\ times}},\ldots,\underbrace{u_n,\ldots,u_n}_{\alpha_n \mathrm{\ times}}); \end{equation*}
\item $\nu(\alpha)$ stands for the number of indices $i$ such that $\alpha_i > 0$.
\end{itemize}
We define $b_n$, $n \geq 0$ as follows: $b_0 = 1$, $b_1 = 1/2$, $b_{2p+1} = 0$ if $p \geq 1$ and $b_{2p} = (-1)^{p-1}B_p,$ with $B_p$ the $p$-th Bernoulli number.
Let
$$
C(W,\alpha):=\left(\frac{1}{\alpha !} \prod_{i=1}^n b_{\alpha_i}\right) K_{\alpha}(W). 
$$

A regular $n$\--wedge is an example of an $n$\--dimensional Delzant polytope.
 Let $\Delta \subset \R^n$ be an $n$\--dimensional  polytope. We say that $\Delta$ is 
 \emph{Delzant} if it is a simple and regular polytope\footnote{$\Delta$ is \emph{simple}
if there are exactly $n$ edges meeting at each vertex of $\Delta$; it is \emph{regular} if the
primitive vectors in the direction of the edges span a basis of $\Z^n$, i.e. 
for each vertex $v$ of $\Delta$,
the edges of $\Delta$ which intersect at $v$ lie on rays
$v +t\alpha_i$, $0 \le t < \infty$, where $\alpha$ is a lattice basis of
$\mathbb{Z}^n$.
}. Suppose that $\Delta$ has $d$ facets. Then
 $\Delta$ is defined by $d$ equations:
$
\langle u_i,\, x \rangle \le c_i,
$
where 
$i \in \{1,\ldots,d\}$. For $q
\in \llbracket 1,n \rrbracket$, we denote by $\mathcal{F}_{q}$ the set of
faces of codimension $q$ of $\Delta$. 

The following theorem gives asymptotic Euler\--MacLaurin formulas for Riemann sums. It holds in any dimension for wedges, and in dimensions $1$, $2$, and $3$ for polytopes.
The case of $4$\--dimensional polytopes is more complicated to handle with our techniques, and we leave it to future works (see Remark~\ref{xx}).  The first assertion of the theorem is the same as~\cite[Proposition~3.1]{Ta2010} (see the Appendix for a comparison), and the second one is similar 
to~\cite[Theorem~5.1]{Ta2010}. In fact, the uniqueness result~\cite[Theorem~5.3]{Ta2010} implies that the operators that we construct are the same, but what differs is the way we obtain them.

\begin{theorem} \label{keyT}
Let $f \in \op{C}^{\infty}_0(\mathbb{R}^n)$. Then the following hold.
\begin{itemize}
\item[(i)]
If $W$ if a regular $n$\--dimensional wedge in $\R^n$,
$$
\frac{1}{N^n} \sum_{k \in \mathbb{Z}^n \cap N W} f \Big( \frac{k}{N} \Big) \sim  \sum_{q \geq 0} N^{-q}\sum_{\alpha \in \mathbb{N}^n \atop |\alpha|=q} 
C(W,\alpha)
\varint_{ \bigcap_{1 \le i \leq n \atop \alpha_i > 0} \mathcal{H}_i } {\rm D}^{q - \nu(\alpha)} f \cdot v_{\alpha - r(\alpha)}, 
$$
where the integral is taken over $W$ if the intersection is empty, and the integral over a single point means evaluation at this point. The sign $\sim$ indicates equality modulo $\mathcal{O}(N^{-\infty})$.
\item[(ii)]
If $\Delta \subset \R^n$, $n \in \{1,2,3\}$, is an $n$\--dimensional Delzant polytope  with vertices in $\Z^n$,
for every $q \geq 1$ and every face $F \in \mathcal{F}_{m}$ with $m \leq q$, 
there exists a linear differential operator
$R_{q}(F,.)$ of degree $q-m$ depending only on $F$ and involving only
derivatives of $f$ in directions normal to $F$ such that
$$
\frac{1}{N^n} \sum_{k \in \mathbb{Z}^n \cap N\Delta} f \Big( \frac{k}{N}
\Big) \sim \varint_{\Delta} f + \sum_{q \geq 1} N^{-q}  \sum_{1 \leq m\leq q \atop F \in \mathcal{F}_{m}}
\varint_{F}^{\star} R_{q}(F,f).
$$
\end{itemize}
\end{theorem}

Theorem~\ref{keyT} will follow from combining several upcoming results: 
\begin{itemize}
\item
Proposition~\ref{keylemma};
\item
Theorem~\ref{theo:w}; 
\item 
Theorem~\ref{h2}.
\end{itemize}
Some of the results of the
paper are more general than Theorem~\ref{keyT}, but we leave them
to later sections for simplicity. Our proof of Theorem~\ref{keyT} is different from the proof of
Tate's general result in \cite[Theorem~5.1]{Ta2010}, self\--contained, and elementary 
(in the sense that it relies only on freshman  calculus). We expect to extend part (ii) of
Theorem~\ref{keyT} to higher dimensions in future works.

In the case where $\Delta$ is a polygon, we give concrete expressions for the coefficients in the formula in Theorem~\ref{keyT}, see Theorem~\ref{lemma:polydim2}.

\begin{figure}
\begin{center}
 \includegraphics[width=4cm]{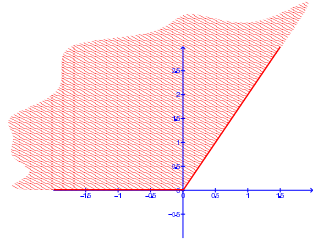}
\end{center}
 \caption{A $2$\--dimensional wedge.}
 \label{fig:wedge}
\end{figure}

\section{Guillemin\--Sternberg formulas for regular wedges and Delzant polytopes}

\subsection{Formula for regular wedges} 
The following approximation result was recently proven by 
Guillemin\--Sternberg \cite{GuSt2007}.

\begin{lemma}
Let $W$ be a regular integer $m$\--wedge defined by \textup{(\ref{for:1})}.
Let $W_h$ be the perturbed set defined by
$
\langle u_i,\, x \rangle \le c_i+h_i$, $i \in \{1,\ldots,m\}$,
where $h=(h_1,\ldots,h_m)$. Then, if $f \in \op{C}^{\infty}_0(\mathbb{R}^n)$,
we have that
$$
\frac{1}{N^n} \sum_{k \in \mathbb{Z}^n \cap N W} f \Big( \frac{k}{N} \Big) \sim
\Big( \tau\Big(   \frac{1}{N} \frac{\partial}{\partial h}\Big) \varint_{W_h}
f(x) \op{d}\!x\Big)(h=0),
$$
where $$\tau(s_1,\ldots,s_m)=\tau(s_1) \ldots \tau(s_m)$$ and $\tau(s_i)$
is the Todd function, on the variable $s_i$, for every $i \in \{1,\ldots, m\}$. The sign $\sim$ indicates equality modulo $\mathcal{O}(N^{-\infty})$.
\label{lm:wedgeGS}
\end{lemma}

\subsection{Formula for Delzant (i.e. regular and  simple) polytopes}
Now let $\Delta \subset \mathbb{R}^n$ be an $n$\--dimensional Delzant polytope 
with vertices in $\mathbb{Z}^n$ and 
exactly $d$ facets. Then $\Delta$ is defined by the inequalities
$
\langle u^i,\, x \rangle \le c_i,\,\,\,\, i \in \{1,\ldots,d\}
$
where $c_i$ is an integer and $u^i \in (\mathbb{Z}^n)^*$ is a
primitive vector perpendicular to the $i^{\op{th}}$\--facet of
$\Delta$, and pointing outwards from $\Delta$. 

Because $\Delta$
is simple by assumption, every codimension $k$ face of $\Delta$
is defined by a collection of equalities
$
\langle u^i,\, x \rangle = c_i$, $i \in F$,
where $F$ is a subset of $k$ elements of the set $\{1,\ldots,d\}$. 
Let $W_F$ denote the $k$\--wedge 
$
\langle u^i,\, x \rangle \le c_i,\,\,\,\, i \in F.
$
Because $\Delta$ is regular, each $k$\--wedge  $W_F$ is regular. 

Guillemin and Sternberg have recently shown \cite{GuSt2007} the following Euler\--MacLaurin formula.

\begin{theorem} \label{regsim:GS}
Let $\Delta$ be a Delzant polytope with vertices in $\Z^n$. Let $\Delta_h$ be the perturbed
polytope defined by the equations 
$
\langle u^i,\, x \rangle \le c_i+h_i$, $i \in \{1,\ldots,d\}$.
Then, if $f \in \op{C}^{\infty}_0(\mathbb{R}^n)$, we have
$$
\frac{1}{N^n}\sum_{k \in \mathbb{Z}^n \cap N \Delta} f\Big( \frac{k}{N}\Big) \sim
\Big( \tau\Big(   \frac{1}{N} \frac{\partial}{\partial h}\Big) \varint_{\Delta_h}
f(x) \op{d}\!x\Big)(h=0),
$$
where 
$$
\tau(s_1,\ldots,s_d)=\tau(s_1) \ldots \tau(s_d)
$$
and $\tau(s_i)$
is the Todd function on the variable $s_i$, for every $i \in \{1,\ldots, d\}$.
\end{theorem}

\begin{remark}
A general asymptotic Euler\--MacLaurin
formula for Riemann sums over lattice polytopes (simple or not) was given by Tate \cite{Ta2010}. As far as we know,
Theorem \ref{regsim:GS} does not follow from Tate's formula.  Other
Euler\--MacLaurin formulas have been obtained by Berline-Vergne 
\cite{BeVe2007}, Brion\--Vergne \cite{BrVe1997}, Cappell\--Shaneson \cite{CS}, 
Karshon\--Sternberg\--Weitsman \cite{KSW}, and Zelditch \cite{Ze2009} among other authors.
\end{remark}

\section{Asymptotic expansion for regular wedges}

Let $W$ be a regular integer $n$\--wedge defined by \textup{(\ref{for:1})}. Recall that $(v_i)_{1 \leq i \leq n}$ is 
the dual basis of $(u_i)_{1 \leq i \leq n}$ and $\mathcal{H}_i$ is the facet of $W$ defined by 
\begin{equation*} \mathcal{H}_i = \left\{ x \in W\,\, | \,\, \langle u_i,x \rangle = c_i \right\}. \end{equation*}
Then $v_{i}$ generates the edge $\bigcap_{j\neq i}  \mathcal{H}_j$. For any integer $q$, we introduce the operator $T_q(W,.)$ defined by
\begin{equation*} T_q(W,f) = \sum_{\alpha \in \mathbb{N}^n \atop |\alpha|=q} \left(\frac{1}{\alpha !} \prod_{i=1}^n b_{\alpha_i}\right) K_{\alpha}(W) \varint_{ \bigcap_{i, \alpha_i > 0} \mathcal{H}_i } {\rm D}^{q - \nu(\alpha)} f \cdot v_{\alpha - r(\alpha)} \end{equation*}
with the convention that the integral is taken over $W$ if the intersection is empty, and that integrating a function over the vertex means evaluating it at the vertex. $K_{\alpha}(W)$ is a constant depending only on the face $\bigcap_{i, \alpha_i > 0} \mathcal{H}_i $, of which we gave an interpretation earlier. 

\begin{prop} \label{keylemma}
If $f \in \op{C}^{\infty}_0(\mathbb{R}^n)$, we have that
\begin{equation} \frac{1}{N^n} \sum_{k \in \mathbb{Z}^n \cap N W} f \Big( \frac{k}{N} \Big) \sim \sum_{q \geq 0} N^{-q} T_{q}(f). \label{eq:expansionwedge} \end{equation}
\end{prop}

\begin{proof}

The idea of the proof is the following: first, we compute the full asymptotic expansion given by Lemma \ref{lm:wedgeGS} in the case of the standard $n$-wedge $\left\{ x \in \mathbb{R}^n; x_1 \leq 0, \ldots, x_n \leq 0 \right\}$. Then, we perform a change of variables to deal with the case of a general regular $n$-wedge.

We start by writing the expansion of $\tau(s_1,\ldots,s_n) = \tau(s_1) \ldots \tau(s_n)$; we recall that 
\begin{equation*} \tau(s) = \frac{s}{1-\exp(-s)} = \sum_{n=0}^{+\infty} b_n \frac{s^n}{n!}. \end{equation*} 
Write
\begin{equation*} \tau(s_1,\ldots,s_n) = \sum_{\alpha \in \mathbb{N}^n} a_{\alpha} s^{\alpha},  \end{equation*}
where $s^{\alpha} = s_1^{\alpha_1} \ldots s_n^{\alpha_n}$. The coefficient $S_q$ of $N^{-q}$ in $\tau\left(\frac{\partial}{\partial h}\right)$ is equal to
\begin{equation*} \sum_{\alpha \in \mathbb{N}^n, |\alpha|=q} a_{\alpha} \frac{\partial^{\alpha}}{\partial h^{\alpha}} \end{equation*} 
with $$\frac{\partial^{\alpha}}{\partial h^{\alpha}} = \frac{\partial^{\alpha_1}}{\partial h_1^{\alpha_1}} \ldots \frac{\partial^{\alpha_n}f}{\partial h_n^{\alpha_n}}.$$
Hence
\begin{equation*} S_q(f) = \sum_{\alpha \in \mathbb{N}^n, |\alpha|=q} \left(\frac{1}{\alpha !} \prod_{i=1}^n b_{\alpha_i}\right) \frac{\partial^{\alpha}}{\partial h^{\alpha}}. \end{equation*}

From this result, we deduce a formula for the case of the standard wedge. Remember that
\begin{equation*} \varint_{W_h} f(x) {\rm d}x = \varint_{-\infty}^{h_1} \ldots \varint_{-\infty}^{h_n} f(x_1,\ldots,x_n) {\rm d}x_1 \ldots {\rm d}x_n; \end{equation*}
thus
\begin{equation*} \left(\frac{\partial}{\partial h_i} \varint_{W_h} f(x) {\rm d}x \right)_{|h=0} = \varint_{-\infty}^{0} \ldots \varint_{-\infty}^{0} f(\widehat{x_i}) \widehat{{\rm d}x_{i}} = \varint_{\{x_i = 0\}} f \end{equation*}
where $\widehat{x_i} = (x_1,\ldots x_{i-1},0,x_{i+1},\ldots,x_n)$ and $\widehat{{\rm d}x_{i}} = {\rm d}x_{1} \ldots {\rm d}x_{i-1} {\rm d}x_{i+1} \ldots {\rm d}x_{n}$. From this we obtain a general formula when $\alpha_i \geq 1$:
\begin{equation*} \left(\frac{\partial^{\alpha_i}}{\partial h_i^{\alpha_i}} \varint_{W_h} f(x) {\rm d}x \right)_{|h=0} = \varint_{\{x_i = 0\}} \frac{\partial^{\alpha_i-1}f}{\partial x_i^{\alpha_i-1}}. \end{equation*}
This finally yields
\begin{equation} \frac{\partial^{\alpha}}{\partial h^{\alpha}} \varint_{W_h} f(x) {\rm d}x = \varint_{\bigcap_{i, \alpha_i > 0} \{x_i =0\}} \frac{\partial^{\alpha-r(\alpha)}f}{\partial x^{\alpha-r(\alpha)}}.  \label{eq:standard}\end{equation}
This gives the desired formula in the case of the standard wedge.

Let us now turn to the general case. Let $W$ be the regular $n$-wedge defined by (\ref{for:1}). Define the diffeomorphism $\varphi: \mathbb{R}^n \to \mathbb{R}^n$, such that $\varphi(x) = y$ has the following coordinates in the standard orthonormal basis of $\R^n$:
\begin{equation*} \forall i \in \{1,\ldots,n\} \quad y_i = \langle u_i,x \rangle - c_i .\end{equation*}
Then $\varphi(W)$ is the standard wedge; moreover, $\varphi$ is a diffeomorphism from $\bigcap_{i \in I} \mathcal{H}_i$ to $\bigcap_{i \in I} \{x_i = 0\}$ for each subset $I$ of $\llbracket 1,n \rrbracket$. We have that 
\begin{equation*} \sum_{k \in \mathbb{Z}^n \cap NW} f\left(\frac{k}{N}\right) = \sum_{\ell \in \varphi(\mathbb{Z}^n) \cap N\varphi(W)} g\left(\frac{\ell}{N}\right) \overset{c_i \ \text{integers}}{=} \sum_{\ell \in \mathbb{Z}^n \cap N\varphi(W)} g\left(\frac{\ell}{N}\right)  \end{equation*}  
with $g = f \circ \varphi^{-1}$. But we know from the previous case that 
\begin{equation*} \sum_{\ell \in \mathbb{Z}^n \cap N\varphi(W)} g\left(\frac{\ell}{N}\right) \sim \sum_{q \geq 0} N^{-q} \ T_q(\varphi(W),g); \end{equation*}
therefore, it only remains to prove that for every $q \geq 0$, $$T_q(\varphi(W),g) = T_q(W,f).$$ First, we have to express the quantity $\frac{\partial^m g}{\partial y_i^m}$ in terms of $f$. Since $g= f \circ \varphi^{-1}$, we have
\begin{equation*} \frac{\partial g}{\partial y_i}(y) = {\rm D} f(\varphi^{-1}(y)) \cdot {\rm D}_{y_i}(\varphi^{-1})(y) ;  \end{equation*}
but $$\varphi^{-1}(y) = \sum_{j=1}^n (y_j + c_j)v_j,$$ so ${\rm D}_{y_i}(\varphi^{-1})(y) = v_i$. It follows that 
\begin{equation*} \frac{\partial g}{\partial y_i}(y) = {\rm D}f(\varphi^{-1}(y)) \cdot v_i. \end{equation*}
By induction, we find
\begin{equation*} \frac{\partial^m g}{\partial y_i^m}(y) = {\rm D}^mf(\varphi^{-1}(y)) \cdot (v_i,\ldots,v_i). \end{equation*}
Now, we have to understand integrals of the form
\begin{equation*} I = \varint_{\bigcap_{i, \alpha_i > 0} \{y_i =0\}} \frac{\partial^{\alpha-r(\alpha)}g}{\partial y^{\alpha-r(\alpha)}}. \end{equation*}
From the previous discussion
\begin{equation*} I = \varint_{\varphi\left( \bigcap_{i, \alpha_i > 0} \mathcal{H}_i \right)} {\rm D}^{|\alpha - r(\alpha)|} f(\varphi^{-1}(y)) \cdot v_{\alpha - r(\alpha)}. \end{equation*}
Hence
\begin{equation*} I = K_{\alpha}(W) \varint_{ \bigcap_{i, \alpha_i > 0} \mathcal{H}_i } {\rm D}^{q - \nu(\alpha)} f \cdot v_{\alpha - r(\alpha)} \end{equation*}
where $K_{\alpha}(W)$ is the Jacobian of the diffeomorphism $\varphi_{|\bigcap_{i, \alpha_i > 0} \mathcal{H}_i}$, which was to be proved.
\end{proof}

Another way to write this result is the following: we can eliminate the constants $K_{\alpha}(W)$ by normalizing the measure on the face $F = \bigcap_{i, \alpha_i > 0} \mathcal{H}_i$; we set 
$$\varint_{F}^{\star} = K_{\alpha}(W) \varint_{F}.$$

\section{Explicit asymptotic expansion for regular wedges}
\label{subsection:expwedge}

In order to deduce a formula for polygons and $3$\--dimensional polytopes, we need to rewrite the asymptotic expansion in
Proposition~\ref{keylemma} in a suitable form. Some of the results of this section apply for general $n\geq 1$,
and when this is the case we present the most general version of the result. The final results require that we restrict to $n=1,2$, or $3$.

\subsection{General results}

Recall that for $d \in \llbracket 1,n \rrbracket$, 
$\mathcal{F}_{d}$ denotes the set of faces of codimension $d$ of $W$. In order to simplify the notation, we set
\begin{equation*} \lambda_{\alpha} = \frac{1}{\alpha !} \prod_{i=1}^n b_{\alpha_i} \end{equation*}
for $\alpha \in \mathbb{N}^n$. Our objective is to write a formula of the following kind:
\begin{equation*} T_{q}(W,f) = \sum_{d=0}^{n-1} \sum_{F \in \mathcal{F}_{d}} S_{q}(F,f)  \end{equation*}
where $S_{q}(F,f)$ is a differential operator associated to the face $F$, with good properties in a sense that we will precise later. Let $F \in \mathcal{F}_{d}$. There exists a subset $I = \{i_{1}, \ldots, i_{d}\}$ of $\llbracket 1,n \rrbracket$ such that $F = \bigcap_{j \in I} \mathcal{H}_{j}$; the family $(v_{i})_{i \notin I}$ is a basis of the linear subspace spanned by $F$. A first expression for $S_{q}(F,f)$ is 
\begin{equation*} S_{q}(F,f) = \sum_{\substack{\alpha \in \mathbb{N}^n,|\alpha|=q\\ \{i, \ \alpha_{i} > 0\} = I}} \lambda_{\alpha}  \varint^{\star}_{F} {\rm D}^{q-d} f \cdot v_{\alpha - r(\alpha)}. \end{equation*}
Observe that if $d > q$, then $S_{q}(F,f) = 0$. If $d=q$, then $S_{q}(F,f)$ involves the integral 
$$\varint^{\star}_{F} f.$$ When $d < q$, the situation is a little bit more complicated, because we integrate directional derivatives of $f$ involving the vectors $v_{i}, i \in I$. But we would like to keep only quantities that depend on the face $F$ and nothing else; this is why we decompose the vectors $v_{i}, i \in I$ as follows:
\begin{equation} v_i = \sum_{j \notin I} \mu_{ij}^F v_j + \sum_{j \in I} \zeta_{ij}^F n_j \label{eq:vi}\end{equation}
where $n_j$ is the outward primitive normal to the facet $\mathcal{H}_j$. Next, we expand the quantity ${\rm D}^{q-d} f \cdot v_{\alpha - r(\alpha)}$ as a linear combination of $n^{q-d}$ terms involving the vectors $v_{j}, j \notin I$ and $n_{j}, j \in I$.

More precisely, write equation (\ref{eq:vi}) as 
\begin{equation*} v_{i} = \sum_{j=1}^n \lambda_{ij}^{F} w_{j} \end{equation*}
where $\lambda_{ij}^F = \mu_{ij}^F$, $w_{j} = v_{j}$ if $j \notin I$ and $\lambda_{ij}^F = \zeta_{ij}^F$, $w_{j} = n_{j}$ if $j \in I$. Moreover, set $\beta = \alpha - r(\alpha)$ and define integers $k_{\ell}$, $1 \leq \ell \leq q-d$ as follows: $k_{\ell} = i_{1}$ if $1 \leq \ell \leq \beta_{i_{1}}$, \ldots, $k_{\ell} = i_{d}$ if $q-d-\beta_{i_{d}} \leq \ell \leq q-d$. Then 
\begin{equation*} {\rm D}^{q-d} f \cdot v_{\alpha - r(\alpha)} = {\rm D}^{q-d} f \cdot \left( \sum_{j_{1}=1}^n \lambda_{k_{1}j_{1}}^F w_{j_{1}}, \ldots,  \sum_{j_{q-d}=1}^n \lambda_{k_{q-d}j_{q-d}}^F w_{j_{q-d}} \right) \end{equation*}
which can be written by multilinearity as 
\begin{equation*} {\rm D}^{q-d} f \cdot v_{\alpha - r(\alpha)} = \sum_{j_{1},\ldots,j_{q-d}=1}^n \left(\prod_{\ell = 1}^{q-d} \lambda_{k_{\ell}j_{\ell}}^F\right) {\rm D}^{q-d} f \cdot \left(  w_{j_{1}}, \ldots,  w_{j_{q-d}} \right). \end{equation*}
We now want to get rid of the vectors $v_{j}$, $j \notin I$ when they appear in the quantity 
$${\rm D}^{q-d} f \cdot \left( w_{j_{1}}, \ldots,   w_{j_{q-d}} \right).$$ If $j_{1}, \ldots, j_{q-d}$ all belong to $I$, then only the normal vectors $n_{j}$ appear, and we have nothing to do; this constitutes $d^{q-d}$ favorable cases. In unfavorable cases, we use the following lemma.
\begin{lemma}
\label{lemma:int}
For any function $g \in \op{C}^{\infty}_0(\mathbb{R}^n)$ and for every $j \notin I$, we have
\begin{equation*} \varint^{\star}_{F} {\rm D}g \cdot v_{j} = \varint^{\star}_{F_{j}} g \end{equation*}
where $F_{j}$ is the face of codimension $d + 1$ defined by $F_{j} = F \cap \mathcal{H}_{j}$.
\end{lemma}
\begin{proof}
Let the elements of $\llbracket 1,n \rrbracket \setminus I$ be denoted by $x_{j},x_{i_{d+2}}, \ldots, x_{i_{n}}$. One has
\begin{eqnarray*}
\varint^{\star}_{F} {\rm D}g \cdot v_{j} &=& \varint_{\{x_{i_{1}} = \cdots = x_{i_{d}} = 0\}} {\rm D}g(\varphi^{-1}(x)) \cdot v_{j} \ {\rm d}x_{j}{\rm d}x_{i_{d+2}} \ldots {\rm d}x_{i_{n}} \\
&=& \varint_{\{x_{i_{1}} = \cdots = x_{i_{d}} = 0\}} {\rm D}g(\varphi^{-1}(x)) \cdot {\rm D}_{x_{j}}(\varphi^{-1})(x) \ {\rm d}x_{j}{\rm d}x_{i_{d+2}} \ldots {\rm d}x_{i_{n}} \\
&=& \varint_{\{x_{i_{1}} = \cdots = x_{i_{d}} = 0\}} \frac{\partial}{\partial x_{j}}(g \circ \varphi^{-1})(x) \ {\rm d}x_{j}{\rm d}x_{i_{d+2}} \ldots {\rm d}x_{i_{n}} \\
&=& \varint_{\{x_{i_{1}} = \cdots = x_{i_{d}} = x_{j} = 0\}} (g \circ \varphi^{-1})(x) \ {\rm d}x_{i_{d+2}} \ldots {\rm d}x_{i_{n}} 
= \varint^{\star}_{F_{j}} g.
\end{eqnarray*}
\end{proof}

\subsection{Results for regular $2$ and $3$\--dimensional wedges}

When $n=1,2,3$, applying Lemma~\ref{lemma:int} to functions of the form $$g = {\rm D}^{q-d-1}f,$$ and repeating this as many times as necessary, we can get rid of all the vectors $v_{j}$, $j \notin I$ in the expression of 
$$\varint^{\star}_{F} {\rm D}^{q-d} f \cdot v_{\alpha - r(\alpha)},$$ and keep only integrals over faces of codimension greater than $d$ of derivatives of $f$ applied to vectors that are normal to the hyperplanes defining the faces.

Before we state and prove our result, let us introduce some useful notation. Let $C(W,F)$ be the cone generated by the set $$\{x-y, y \in W, x \in F\}.$$ If $X$ is a non-empty subset of $\mathbb{R}^n$, let $L(X)$ be the vector subspace generated by the elements of the form $y-x$, $x,y \in X$.

The following two lemmas hold in any dimension (not just $2$ and $3$).

\begin{lemma}
\label{lemma:depend}
Decompose the vectors $v_{i}$, $i \in I$, as in equation (\ref{eq:vi}). Choose another set of vectors $(w_{i})_{i \in I}$ such that	
\begin{itemize}
\item $\forall i \in I$, $w_{i}$ belongs to $$L\left(\bigcap_{j \in \llbracket 1,n \rrbracket \setminus\{i\}} \mathcal{H}_{j}\right) \cap C(W,F),$$
\item the family $((v_{j})_{j \notin I},(w_{j})_{j \in I})$ is a primitive lattice basis,
\end{itemize}
and write, for $i \in I$,
\begin{equation*} w_i = \sum_{j \notin I} \tilde{\mu}_{ij}^F v_j + \sum_{j \in I} \tilde{\zeta}_{ij}^F n_j.\end{equation*}
Then for $j \in I$, we have $\tilde{\zeta}_{ij}^F = \zeta_{ij}^F$.
\end{lemma}
In other words, this means that the scalars $\zeta_{ij}^F$ only depend on the face $F$; when $F$ will be considered as a face of a polytope instead of a wedge, then the contribution coming from each wedge will display the same coefficient. 
\begin{proof}
By definition of the vectors $v_{j}$, $n_{j}$, $1 \leq j \leq n$, one has 
\begin{equation} \forall \ell \in I \qquad \langle v_{i},n_{\ell} \rangle = \sum_{j \notin I} \zeta_{ij}^F \langle n_{j},n_{\ell} \rangle. \label{eq:zeta}\end{equation}
Hence, the $d$ coefficients $\zeta_{ij}^F$, $j \in I$, are obtained by solving the linear system (\ref{eq:zeta}) of $d$ equations. Now, we express the vector $w_{i}$ in the basis $\mathcal{B}=(v_{j})_{1 \leq j \leq n}$:
\begin{equation*} w_{i} = \sum_{j \notin I} \alpha_{ij} v_{j} + \sum_{j \in I} \beta_{ij} v_{j}. \end{equation*}
Since the vector $w_{i}$ belongs to $L(\bigcap_{j \in \llbracket 1,n \rrbracket \setminus\{i\}} \mathcal{H}_{j})$, all the scalar products $\langle w_{i}, n_{j} \rangle$, $j \in I \setminus\{i\}$, vanish. This implies that for every $j \in I \setminus\{i\}$, $\beta_{ij} = 0$. Thus, the matrix $M$ of change of basis from $\mathcal{B}$ to 
$$\mathcal{B}'=((v_{j})_{j \notin I},(w_{j})_{j \in I})$$ is of the form
\begin{equation*} M = \begin{pmatrix} I_{n-d} & A \\ 0 & B \end{pmatrix} \end{equation*}
where $$B = \text{diag}(\beta_{11},\ldots,\beta_{dd}).$$ Since $\mathcal{B}$ and $\mathcal{B}'$ are primitive lattice bases, we have $\det(M) = \pm 1$, and hence for every $i \in I$, $\beta_{ii} = \pm 1$. But $v_{i}$ and $w_{i}$ belong to $C(W,F)$, so $\beta_{ii} = 1$. This yields that for $\ell \in I$, we have $\langle w_{i},n_{\ell} \rangle = \langle v_{i},n_{\ell} \rangle$.
\end{proof}

\begin{lemma}
\label{lemma:dependmu}
Decompose the vectors $v_{i}$, $i \in I$, as in equation (\ref{eq:vi}), and fix $j \notin I$. Set $J = I \cup \{j\}$, and choose another set of vectors $(w_{i})_{i \in J}$ such that	
\begin{itemize}
\item $\forall i \in J$, $w_{i}$ belongs to $$L\left(\bigcap_{k \in \llbracket 1,n \rrbracket \setminus\{i\}} \mathcal{H}_{k} \right) \cap C(W,F),$$
\item the family $((v_{i})_{i \notin J},(w_{i})_{i \in J})$ is a primitive lattice basis.
\end{itemize}
and write, for $i \in I$,
\begin{equation*} w_i = \tilde{\mu}_{ij}^F w_j + \sum_{k \notin J} \tilde{\mu}_{ik}^F v_k + \sum_{k \in I} \tilde{\zeta}_{ik}^F n_k.\end{equation*}
Then we have $\tilde{\mu}_{ij}^F = \mu_{ij}^F$.
\end{lemma}
In other words, this means that the scalar $\mu_{ij}^F$ only depends on the face $F \cap \mathcal{H}_{j}$. 

\begin{proof}
By definition of the vectors $v_{k}$, $n_{k}$, $1 \leq k \leq n$, one has 
\begin{equation} \forall k \notin I \qquad \langle v_{i},v_{k} \rangle = \sum_{\ell \notin I} \mu_{i\ell}^F \langle v_{\ell},v_{k} \rangle, \label{eq:mu} \nonumber
\end{equation}
which can be written in matrix form
\begin{equation} A\nu = V \label{eq:S1}\end{equation}
where $V, \nu$ are the column vectors given by
\begin{equation*} \forall \ell \notin I \qquad V_{\ell} = \langle v_{i},v_{\ell} \rangle, \quad \nu_{\ell} = \mu_{i\ell}^F  \end{equation*}
and $A$ is the symmetric matrix whose generic coefficient is $A_{k,\ell} = \langle v_{k}, v_{\ell} \rangle$, $k,\ell \notin I$. Similarly, the constants $\tilde{\mu}_{i\ell}^F$ satisfy the system of equations
\begin{equation} \forall k \notin J \qquad \langle w_{i},v_{k} \rangle = \tilde{\mu}_{ij}^F \langle w_{j},v_{k} \rangle + \sum_{\ell \notin J} \tilde{\mu}_{i\ell}^F \langle v_{\ell},v_{k} \rangle \tag{$E_{k}$}\label{eq:Ek}\end{equation}
and
\begin{equation} \langle w_{i},w_{j} \rangle = \tilde{\mu}_{ij}^F \|w_{j}\|^2 + \sum_{\ell \notin J} \tilde{\mu}_{i\ell}^F \langle v_{\ell},w_{j} \rangle. \tag{$E_{i}$}\label{eq:Ei}\end{equation}
Thanks to the proof of the previous lemma, we know that there exists scalars $\alpha_{k\ell}$, $k \in J$, $\ell \notin J$ such that 
\begin{equation*} \forall k \in J \qquad w_{k} = v_{k} + \sum_{\ell \notin J} \alpha_{k\ell} v_{\ell}. \end{equation*}
Hence we have
\begin{equation*} \langle w_{i},w_{j} \rangle = \langle v_{i},v_{j} \rangle + \sum_{\ell \notin J} \alpha_{j\ell} \langle v_{i},v_{\ell} \rangle + \sum_{k \notin J} \alpha_{ik} \langle v_{j},v_{k} \rangle + \sum_{k,\ell \notin J} \alpha_{ik} \alpha_{j\ell} \langle v_{\ell},v_{k} \rangle \end{equation*}
as well as
\begin{eqnarray*} \forall k \notin J \qquad \langle w_{i},v_{k} \rangle &=& \langle v_{i},v_{k} \rangle + \sum_{\ell \notin J} \alpha_{i\ell} \langle v_{k},v_{\ell} \rangle \\ \langle w_{j},v_{k} \rangle &=& \langle v_{j},v_{k} \rangle + \sum_{\ell \notin J} \alpha_{j\ell} \langle v_{k},v_{\ell} \rangle 
\end{eqnarray*}
and 
\begin{equation*} \| w_{j}\|^2 = \|v_{j}\|^2 + 2 \sum_{\ell \notin J} \alpha_{j\ell} + \sum_{k,\ell \notin J} \alpha_{j\ell} \alpha_{jk} \langle v_{\ell},v_{k} \rangle. \end{equation*}
Using these relations, equations (\ref{eq:Ek}) become
\begin{equation} \langle v_{i},v_{k} \rangle + \sum_{\ell \notin J} \alpha_{i\ell} \langle v_{\ell},v_{k} \rangle = \tilde{\mu}_{ij}^F \langle v_{j},v_{k} \rangle +  \sum_{\ell \notin J} \left( \tilde{\mu}_{i\ell}^F + \tilde{\mu}_{ij}^F \alpha_{j\ell} \right) \langle v_{\ell},v_{k} \rangle \tag{$E'_{k}$}\label{eq:E'k}\end{equation}
while equation (\ref{eq:Ei}) becomes
\begin{equation}\begin{split} \langle v_{i},v_{j} \rangle + \sum_{\ell \notin J} \alpha_{j\ell} \langle v_{i},v_{k} \rangle + \sum_{k \notin J} \alpha_{ik} \langle v_{j},v_{k} \rangle + \sum_{k,\ell \notin J} \alpha_{ik} \alpha_{j\ell} \langle v_{k},v_{\ell} \rangle \\ = \tilde{\mu}_{ij}^F \|v_{j}\|^2 
+ 2 \sum_{\ell \notin J} \tilde{\mu}_{ij}^F \alpha_{j\ell} \langle v_{j},v_{\ell} \rangle + \sum_{k,\ell \notin J} \tilde{\mu}_{ij}^F \alpha_{j\ell} \alpha_{jk} \langle v_{k},v_{\ell} \rangle \\+ \sum_{\ell \notin J} \tilde{\mu}_{i\ell}^F \langle v_{\ell},v_{j} \rangle + \sum_{k,\ell \notin J} \tilde{\mu}_{i\ell}^F\alpha_{jk} \langle v_{k},v_{\ell} \rangle. \end{split} \tag{$E'_{i}$}\label{eq:E'i}\end{equation}
Considering the linear combination (\ref{eq:E'i}) - $\sum_{k \notin J} \alpha_{jk}$(\ref{eq:E'k}), we replace equation (\ref{eq:E'i}) by the new equation (we do no write the details of the computations)
\begin{equation*} \langle v_{i},v_{j} \rangle + \sum_{k \notin J} \alpha_{ik} \langle v_{j},v_{k} \rangle = \tilde{\mu}_{ij}^F \|v_{j}\|^2 + \sum_{\ell \notin J} \left( \tilde{\mu}_{i\ell}^F + \tilde{\mu}_{ij}^F \alpha_{j\ell} \right) \langle v_{j},v_{\ell} \rangle. \end{equation*}
Together with equations (\ref{eq:E'k}), this means that the coefficients $\tilde{\mu}_{i\ell}^F$ are solutions of the system
\begin{equation*} AU + V = A\tilde{\nu} \end{equation*}
where $A$ and $V$ are as before, $\tilde{\nu}$ is defined as $\nu$ but with the coefficients $\tilde{\mu}_{i\ell}^F$ instead of $\mu_{i\ell}^F$, and $U$ is the column vector whose entries are $U_{j} = 0$, $$U_{\ell} = \alpha_{i\ell} - \tilde{\mu}_{ij}^F\alpha_{j\ell}.$$ Comparing this to (\ref{eq:S1}) yields $\nu = \tilde{\nu} - U$, and in particular $\mu_{ij}^F = \tilde{\mu}_{ij}^F$.
\end{proof}

\begin{theorem} \label{theo:w}
Assume that $n \in \{1,2,3\}$. For every $q \geq 1$ and every face $F \in \mathcal{F}_{d}$ with $d \leq q$, there exists a linear differential operator $R_{q}(F,.)$ of degree $q-d$ depending only on $F$ (in the sense introduced in the previous lemmas) and involving only derivatives of $f$ in directions normal to the face $F$ such that
\begin{equation} T_{q}(W,f) = \sum_{d=0}^{n-1} \sum_{F \in \mathcal{F}_{d}} \varint_{F}^{\star} R_{q}(F,f).  \end{equation}
\end{theorem}
\begin{proof}
To compute $R_{q}(F,.)$, we apply the previous technique to faces of codimension smaller than $d$ and gather their contribution as integrals over $F$.
It follows from Lemma~\ref{lemma:depend} and Lemma~\ref{lemma:dependmu} that 
$R_{q}(F,.)$ depends only on the face $F$; let us briefly explain how.

If $n=2$, we have to handle two types of faces: the two edges ($d=1$) and the vertex ($d=2$) of the wedge. There is not much to say about the case of the vertex. When we integrate over an edge, and we apply our technique, we will find
\begin{itemize}
\item constants $\mu$ in front of derivatives of $f$ evaluated at the vertex, and there is nothing to prove,
\item constants $\zeta$ in front of integrals of derivatives of $f$ on $F$, and Lemma~\ref{lemma:depend} 
ensures that it only depends on the face $F$.
\end{itemize}

If $n=3$, we have three types of faces, namely planes ($d=1$), edges ($d=2$) and the vertex ($d=3$). The difference with the previous case is that when we consider integrals over a plane, we obtain integrals over edges belonging to this plane, each one displaying a factor $\mu$; 
Lemma~\ref{lemma:dependmu} ensures that it only depends on the given edge.
\end{proof}

\begin{remark} \label{xx}
\begin{enumerate}
\item In dimension $4$ and higher, the situation is  more complicated, and Lemma \ref{lemma:depend} and 
Lemma \ref{lemma:dependmu} are not enough to obtain a similar theorem. Indeed, think of the following situation: we take $n=4$ and want to evaluate $$I = \varint^{\star}_{\mathcal{H}_{1}} {\rm D}^2f \cdot (v_{1},v_{1}).$$ We start by expanding
\begin{equation*} v_{1} = \mu_{12}^{\mathcal{H}_{1}} v_{2} + \mu_{13}^{\mathcal{H}_{1}} v_{3} + \mu_{14}^{\mathcal{H}_{1}} v_{4} + \zeta_{11}^{\mathcal{H}_{1}} n_{1}, \end{equation*} 
and we apply Lemma \ref{lemma:int} to obtain
\begin{eqnarray*} I = \mu_{12}^{\mathcal{H}_{1}} \varint^{\star}_{\mathcal{H}_{1} \cap \mathcal{H}_{2}} {\rm D}f \cdot v_{1} + \mu_{13}^{\mathcal{H}_{1}} \varint^{\star}_{\mathcal{H}_{1} \cap \mathcal{H}_{3}} {\rm D}f \cdot v_{1} \\
+ \mu_{14}^{\mathcal{H}_{1}} \varint^{\star}_{\mathcal{H}_{1} \cap \mathcal{H}_{4}} {\rm D}f \cdot v_{1} + \zeta_{11}^{\mathcal{H}_{1}} \varint^{\star}_{\mathcal{H}_{1}} {\rm D}^2 f \cdot (v_{1},n_{1}). \end{eqnarray*}
We have to apply the method one more time for each of these integrals. For instance, we put 
$$K = \varint^{\star}_{\mathcal{H}_{1} \cap \mathcal{H}_{2}} {\rm D}f \cdot v_{1}$$ and to compute this integral, we write
\begin{equation*} v_{1} = \mu_{13}^{\mathcal{H}_{1} \cap \mathcal{H}_{2}} v_{3} + \mu_{14}^{\mathcal{H}_{1} \cap \mathcal{H}_{2}} v_{4}  + \zeta_{11}^{\mathcal{H}_{1} \cap \mathcal{H}_{2}} n_{1} + \zeta_{12}^{\mathcal{H}_{1} \cap \mathcal{H}_{2}} n_{2} \end{equation*}
which yields, again thanks to Lemma~\ref{lemma:int}
\begin{equation*} \begin{split} K = \mu_{13}^{\mathcal{H}_{1} \cap \mathcal{H}_{2}} \varint^{\star}_{\mathcal{H}_{1} \cap \mathcal{H}_{2} \cap \mathcal{H}_{3}} f + \mu_{14}^{\mathcal{H}_{1} \cap \mathcal{H}_{2}} \varint^{\star}_{\mathcal{H}_{1} \cap \mathcal{H}_{2} \cap \mathcal{H}_{4}} f  + \\ 
\zeta_{11}^{\mathcal{H}_{1} \cap \mathcal{H}_{2}} \varint^{\star}_{\mathcal{H}_{1} \cap \mathcal{H}_{2}} {\rm D}f \cdot n_{1} + \zeta_{12}^{\mathcal{H}_{1} \cap \mathcal{H}_{2}} \varint^{\star}_{\mathcal{H}_{1} \cap \mathcal{H}_{2}} {\rm D}f \cdot n_{2} \end{split}\end{equation*}

Hence, in the expression of $I$, we obtain the term 
$$\mu_{12}^{\mathcal{H}_{1}} \mu_{13}^{\mathcal{H}_{1} \cap \mathcal{H}_{2}} \varint^{\star}_{\mathcal{H}_{1} \cap \mathcal{H}_{2} \cap \mathcal{H}_{3}} f;$$ does the factor only depend on the face $\mathcal{H}_{1} \cap \mathcal{H}_{2} \cap \mathcal{H}_{3}$ ? We think that our previous lemmas are not enough to give an answer to this question.

\item In principle, one should be able to obtain explicit expressions for the operators $R_{q}(F,f)$, but this leads to computations involving a very large number of constants, so it is not reasonable to try to write such an expression. However, if we restrict ourselves to dimension $2$, we can be fully explicit, as we will see later.
\end{enumerate}
\end{remark}

\section{Asymptotic expansion for polygons and $3$\--dimensional polytopes}

As in Section~\ref{sec:intro}, let $\Delta \subset \R^n$ be a Delzant polytope with vertices in $\Z^n$ in dimension $n \in \{1,2,3\}$ with equations
$
\langle u^i,\, x \rangle \le c_i$, $i \in \{1,\ldots,d\}$. 
Recall that for $m \in \llbracket 1,n \rrbracket$,  $\mathcal{F}_{m}$ denotes the set of faces of codimension $m$ of $\Delta$. 

We introduce as in Theorem~\ref{theo:w} the operators $R_{q}(F,.)$ associated to a face $F$ of the polytope (remembering that it only depends on the face as part of the polytope). For any integer $q$, we define the operator $T_q(\Delta,.)$ by
\begin{equation} T_q(\Delta,f) = \sum_{m=0}^{n-1} \sum_{F \in \mathcal{F}_{m}} \varint_{F}^{\star} R_{q}(F,f).  \label{eq:Tq}\end{equation}

\begin{theorem} \label{h2}

If $f \in \op{C}^{\infty}_0(\mathbb{R}^n)$, we have that
$$
\frac{1}{N^n} \sum_{k \in \mathbb{Z}^n \cap N\Delta} f \Big( \frac{k}{N} \Big)= \sum_{q \geq 0} N^{-q} T_{q}(\Delta,f).$$

\label{thm:main}\end{theorem}

\begin{proof}
Notice first that $\Delta = \bigcap_{i=1}^p W_i$ where $p$ is the number of vertices of $\Delta$ and $W_i$ is the regular wedge which is the intersection of the $n$ facets $\mathcal{H}^i_j$, $1 \leq j \leq n$, intersecting at the vertex $v_i$. Cover $\Delta$ by open sets $\Omega_i$, $1 \leq i \leq p$, such that $\Omega_i$ contains the vertex $v_i$ and does not intersect any other facet than the $\mathcal{H}^i_j$, $1 \leq j \leq n$. Choose a partition of unity associated to this open covering and write $f = \sum_{i=1}^p f_i$ where $f_i \in \op{C}^{\infty}_0(\mathbb{R}^n)$ has support included in $\Omega_i$. Then 
\begin{equation*} \sum_{k \in \mathbb{Z}^n \cap N\Delta} f \Big( \frac{k}{N} \Big) = \sum_{i=1}^p \sum_{k \in \mathbb{Z}^n \cap N W_i} f_i \Big( \frac{k}{N} \Big).\end{equation*}
Now, from formula (\ref{eq:expansionwedge}), we know that for $1 \leq i \leq d$
\begin{equation*} \sum_{k \in \mathbb{Z}^n \cap N W_i} f_i \Big( \frac{k}{N} \Big) \sim \sum_{\alpha \geq 0} N^{-q} T_{q}(W_i,f_i); \end{equation*}
hence it is enough to check that for all $q$
\begin{equation*} \sum_{i=1}^p T_{q}(W_i,f_i) = T_{q}(\Delta,f). \end{equation*}
This amounts to show that for each face $F$ of the polytope
\begin{equation*} \sum_{i=1}^p R_{q}(F,f_i) = R_{q}(F,f). \end{equation*}
But this is clear because $R_{q}(F,.)$ is linear and because $\sum_{i=1}^p f_{i} = f$.
\end{proof}

\section{Explicit formula in dimension $n=2$} \label{n2case}

We would like to compute explicitly the operators $R_{q}(\Delta,.)$; unfortunately, as already said, this seems to be quite complicated in all generality. However, we can give nice formulas in dimension $2$.

Let $\Delta$ be a regular integer polygon defined by \textup{(\ref{for:1})}. In this case, we only have two types of faces: vertices (codimension $2$) and edges (codimension $1$). Let $E$ (resp. $V$) be the set of edges (resp. vertices) of $\Delta$. If $e$ belongs to $E$, let $n_e$ be the associated outward primitive normal vector; if $v$ belongs to $V$, let $(w_{1}(v),w_{2}(v))$ be the integral basis of $\mathbb{Z}$ such that the two edges meeting at $v$ are contained in the half-lines $v + \lambda w_{i}(v)$, $\lambda \geq 0$; we denote by $e_{i}$ the edge generated by $w_{i}(v)$. Define the quantities
\begin{equation*} \eta_{1}(v) = \frac{\langle w_{1}(v),w_{2}(v) \rangle}{||w_{1}(v)||^2}, \qquad \eta_{2}(v) = \frac{\langle w_{1}(v),w_{2}(v) \rangle}{||w_{2}(v)||^2}.  \end{equation*}
and $\mu(v) = \eta_{1}(v) + \eta_{2}(v)$.

Now, let $e$ be an edge, and let $C(\Delta,e)$ be the cone generated by the set $\{x-y, y \in \Delta, x \in e\}$. Given a generator $v_{1}$ of $e \cap \mathbb{Z}^2$, there exists a vector $v_{2} \in C(\Delta,e) \cap \mathbb{Z}^2$ such that $(v_{1},v_{2})$ is a primitive lattice basis of $\mathbb{Z}^2$.
\begin{lemma}[Easy version of Lemma \ref{lemma:depend}]
The quantity
\begin{equation*} \zeta(e) = \frac{\langle v_{2},n_{e} \rangle}{||n_{e}||^2} \end{equation*}
does not depend on the choice of $v_{2} \in C(\Delta,e) \cap \mathbb{Z}^2$.
\end{lemma}
\begin{proof}
The lemma follows from Lemma \ref{lemma:depend}, but its proof is very simple, so we present it next.
Choose another vector $w_{2} \in C(\Delta,e) \cap \mathbb{Z}^2$ such that $(v_{1},w_{2})$ is a primitive lattice basis of $\mathbb{Z}^2$. Write $w_{2} = \alpha v_{1} + \beta v_{2}$; then, one has $\langle w_{2},n_{e} \rangle = \beta \langle v_{2},n_{e} \rangle$. The matrix of the change of the basis is of the form 
\begin{equation*} A = \begin{pmatrix} 1 & \alpha \\ 0 & \beta \end{pmatrix}; \end{equation*}
because its determinant must be $\pm 1$, we have $\beta = \pm 1$. Since both $v_{2}$ and $w_{2}$ belong to $C(\Delta,e)$, the only possibility is $\beta =1$.
\end{proof}

\begin{theorem} 
\label{lemma:polydim2}
In Theorem \ref{thm:main} the operators $T_q(\Delta,.)$ are given by:
\begin{itemize}
\item
$T_0(\Delta,f)=\varint_{\Delta} f(x) {\rm d}x$;\\
\item
$T_1(\Delta,f)=\frac{1}{2} \sum_{e \in E} \varint_{e}^{\star} f$;\\
\item
$T_2(\Delta,f)=\sum_{v \in V}\left(\frac{1}{4} {+} \frac{\mu(v)}{12} \right) f(v)-\frac{1}{12} \sum_{e \in E} \zeta(e) \varint_{e}^{\star} {\rm D}f \cdot n_e
$;\\
\item if $p>1$, then
\begin{equation*} T_{2p}(\Delta,f) = \sum_{e \in E} R_{2p}(e,f) + \sum_{v \in V} R_{2p}(v,f) \end{equation*}
where
\begin{equation*}R_{2p}(e,f) = (-1)^{p-1} \frac{B_p}{(2p)!} \zeta(e)^{2p-1}  \varint_{e}^{\star}{\rm D}^{2p-1}f \cdot (n_e,\ldots,n_{e})  \end{equation*}
and $R_{2p}(v,f)$ is equal to
\begin{equation*} \begin{split}   (-1)^{p-2} \sum_{m+\ell=p \atop m,\ell\geq 1}  \frac{B_mB_{\ell}}{(2m)! (2 \ell)!} {\rm D}^{2p-2}f(v) \cdot (\underbrace{w_{1}(v),\ldots,w_{1}(v)}_{2m-1 \,\, {\rm times}},\underbrace{w_{2}(v),\ldots,w_{2}(v)}_{2\ell-1 \,\, {\rm times}})  \\ 
+ (-1)^{p-2} \frac{B_p}{(2p)!} \eta_{1}(v) \sum_{k=0}^{2p-2} \zeta(e_{1})^k {\rm D}^{2p-2}f \cdot (\underbrace{n_{e_{1}},\ldots,n_{e_{1}}}_{k \,\, {\rm times}},\underbrace{w_{2}(v),\ldots,w_{2}(v)}_{2p-2-k \,\, {\rm times}}) \\
+ (-1)^{p-2} \frac{B_p}{(2p)!} \eta_{2}(v) \sum_{k=0}^{2p-2} \zeta(e_{2})^k {\rm D}^{2p-2}f \cdot (\underbrace{n_{e_{2}},\ldots,n_{e_{2}}}_{k \,\, {\rm times}},\underbrace{w_{1}(v),\ldots,w_{1}(v)}_{2p-2-k \,\, {\rm times}});
\end{split} \end{equation*}
\item
if $p>1$, $T_{2p+1}(\Delta,f)$ is equal to
\begin{equation*} 
 \frac{(-1)^{p-1}B_p}{2 (2p)!}   \sum_{v \in V} \left( {\rm D}^{2p-1}f(v) \cdot ({w_1(v),\ldots,w_1(v)}) \\
+ {\rm D}^{2p-1}f(v) \cdot ({w_2(v),\ldots,w_2(v)})   \right).
\end{equation*}
\end{itemize}\end{theorem}

\begin{remark}
Theorem~\ref{lemma:polydim2} recovers the formula  in \cite[Corollary $5.4$]{Ta2010}. To compare the two formulas, one may notice that Tate does not separate the even and odd cases, and that in the odd case nearly every coefficient in Tate's formula vanishes because of the properties of the Bernoulli numbers.
\end{remark}

\begin{proof}

We have to compute the operators $R_{q}(F,.)$ as in Section \ref{subsection:expwedge}. We start by the case $q=2$. Let $v$ be a vertex and let $W$ be the wedge formed by this vertex and the two incident edges. Define the vectors $w_{1}(v),w_{2}(v)$ as before, and let $e_{1}$ (resp. $e_{2}$) be the edge generated by $w_{1}(v)$ (resp. $w_{2}(v)$). 

We have
\begin{equation*} T_{2}(W,f) = \frac{1}{4} f(v) - \frac{1}{12} \left( \varint^{\star}_{e_{1}} {\rm D}f\cdot w_{2}(v) +  \varint^{\star}_{e_{2}} {\rm D}f\cdot w_{1}(v) \right) \end{equation*}
If $n_{i}$ is the outward primitive vector normal to the edge $e_{i}$, we write
\begin{equation*} w_{i}(v) = \alpha_{i} w_{j}(v) + \beta_{i} n_{j}  \end{equation*}
where $j=2$ (resp. 1) if $i=1$ (resp. 2). Taking the scalar product with $n_{j}$ and $w_{j}(v)$, we find
\begin{equation*} \alpha_{i} = \frac{\langle w_{i}(v),w_{j}(v) \rangle}{||w_{j}(v)||^2} = \eta_{j}(v), \qquad \beta_{i} = \frac{\langle w_{i}(v),n_{j} \rangle}{||n_{j}||^2} = \zeta(e_{j}). \end{equation*}
Now, thanks to lemma \ref{lemma:int}, we have {(being careful that the vector $w_i(v)$ is the opposite of the vector $v_j$ in this lemma)}
\begin{equation*} \varint^{\star}_{e_{j}} {\rm D}f\cdot w_{i}(v) = {-}\alpha_{i} f(v) + \zeta(e_{j}) \varint^{\star}_{e_{j}} {\rm D}f\cdot n_{j}.  \end{equation*}
Adding the contributions from each vertex, we obtain the desired formula.

Now, let $p >1$; then 
\begin{equation*} \begin{split} T_{2p}(W,f) = (-1)^{p-1} \frac{B_p}{(2p)!}  \varint_{e_{1}}^{\star}{\rm D}^{2p-1}f \cdot (w_{2}(v),\ldots,w_{2}(v)) \\ + (-1)^{p-1} \frac{B_p}{(2p)!} \varint_{e_{2}}^{\star}{\rm D}^{2p-1}f \cdot (w_{1}(v),\ldots,w_{1}(v))  \\ 
+ (-1)^{p-2} \sum_{m+\ell=p,\atop m,\ell\geq 1}  \frac{B_mB_{\ell}}{(2m)! (2 \ell)!}  {\rm D}^{2p-2}f(v) \cdot (\underbrace{w_{1}(v),\ldots,w_{1}(v)}_{2m-1 \,\, {\rm times}},\underbrace{w_{2}(v),\ldots,w_{2}(v)}_{2\ell-1 \,\, {\rm times}}). \end{split} \end{equation*}
We write
\begin{eqnarray*} \begin{split} \varint_{e_{j}}^{\star}{\rm D}^{2p-1}f \cdot (w_{i}(v),\ldots,w_{i}(v)) 
&=& \eta_{j}(v) \varint_{e_{j}}^{\star}{\rm D}^{2p-1}f \cdot (w_{j}(v),w_{i}(v),\ldots,w_{i}(v)) \\ 
&&+ \zeta(e_{j}) \varint_{e_{j}}^{\star}{\rm D}^{2p-1}f \cdot (n_{j},w_{i}(v),\ldots,w_{i}(v)). \end{split}
\end{eqnarray*}
By Lemma \ref{lemma:int}, we have 
\begin{equation*} 
\varint_{e_{j}}^{\star}{\rm D}^{2p-1}f \cdot (w_{j}(v),w_{i}(v),\ldots,w_{i}(v)) = {-} {\rm D}^{2p-2}f \cdot (w_{i}(v),\ldots,w_{i}(v))
\end{equation*}
and hence we obtain 
\begin{equation*} \begin{split} \varint_{e_{j}}^{\star}{\rm D}^{2p-1}f \cdot (w_{i}(v),\ldots,w_{i}(v)) = {-}\eta_{j}(v) {\rm D}^{2p-2}f \cdot (w_{i}(v),\ldots,w_{i}(v)) \\ + \zeta(e_{j}) \varint_{e_{j}}^{\star}{\rm D}^{2p-1}f \cdot (n_{j},w_{i}(v),\ldots,w_{i}(v)). \end{split}
\end{equation*}
By a straightforward induction, this yields 
\begin{equation*} \begin{split} \varint_{e_{j}}^{\star}{\rm D}^{2p-1}f \cdot (w_{i}(v),\ldots,w_{i}(v)) \\= {-} \eta_{j}(v) \sum_{k=0}^{2p-2} \zeta(e_{j})^k {\rm D}^{2p-2}f \cdot (\underbrace{n_{j},\ldots,n_{j}}_{k \,\, {\rm times}},\underbrace{w_{i}(v),\ldots,w_{i}(v)}_{2p-2-k \,\, {\rm times}}) \\ + \zeta(e_{j})^{2p-1} \varint_{e_{j}}^{\star}{\rm D}^{2p-1}f \cdot (n_{j},\ldots,n_{j}). \end{split}  \end{equation*}

The case $q = 2p+1$ works in a  similar way.
\end{proof}

\section{Examples}

Let us describe two examples in the $2$-dimensional case. The first one is fully explicit, in the sense that we choose a function for which the asymptotic expansion given in Theorem \ref{thm:main} is finite, and we provide every coefficient of the latter. The second one is more involved and we compute the first three terms of the asymptotic expansion and give numerical computations of the remainder.

\subsection{A fully explicit example}

Let $\Delta$ be the triangle with vertices $(0,0)$, $(0,1)$ and $(1,0)$ and let $$f:\mathbb{R}^2 \to \mathbb{R}$$ be the function defined by 
$$f(x_1,x_2) = x_1$$ (multiplied by a cutoff function so that it is compactly supported). Then
\begin{equation*} \mathbb{Z}^2 \cap N\Delta = \{ (k_{1},k_{2}) \in \mathbb{Z}^2, \quad 0 \leq k_{1} \leq N, \quad 0 \leq k_{2} \leq N-k_{1} \}. \end{equation*} 
Therefore, we have 
\begin{equation*} \frac{1}{N^2} \sum_{k \in N\Delta \cap \mathbb{Z}^2} f\left(\frac{k}{N}\right) = \frac{1}{N^2} \sum_{k_{1}=0}^N \sum_{k_{2} = 0}^{N-k_{1}} \frac{k_{1}}{N} = \frac{1}{N^3} \left( (N+1) \sum_{k_{1}=1}^N k_{1} - \sum_{k_{1}=1}^N k_{1}^2 \right). \end{equation*}
Using standard formulas for sums of integers and squares of integers, one can check that 
\begin{equation*} \frac{1}{N^2} \sum_{k \in N\Delta \cap \mathbb{Z}^2} f\left(\frac{k}{N}\right) = \frac{1}{6} + \frac{1}{2N} + \frac{1}{3N^2} . \end{equation*}
Let us compare this with Theorem~\ref{lemma:polydim2}. With the notation of this lemma, we have
\begin{equation*} T_{0}(\Delta,f) = \varint_{\Delta} f = \varint_{0}^1 \left( \varint_{0}^{1-x_{1}} {\rm d}x_{2} \right) x_{1} {\rm d}x_{1} = \frac{1}{6} \end{equation*}
so the zeroth order terms agree. Let us give names to the vertices and edges of $\Delta$ as follows: we put $v_{13} = (0,0)$, $v_{12} = (0,1)$ and $v_{23} = (1,0)$, and we let $e_{i}$ denote the edge joining the vertices $v_{ij}$ (or $v_{ji}$) and $v_{ik}$ (or $v_{ki}$). Then we have 
\begin{equation*} \varint^{\star}_{e_{1}} f = 0, \quad \varint^{\star}_{e_{2}} f = \frac{1}{2}, \quad \varint^{\star}_{e_{3}} f = \frac{1}{2}. \end{equation*}
and hence $$T_{1}(\Delta,f) = 1/2.$$ Furthermore, we have $$T_{2}(\Delta,f) = S - T$$ with 
\begin{equation*} S = \sum_{v \in V}\left(\frac{1}{4} {+} \frac{\mu(v)}{12} \right) f(v),\end{equation*}
and
\begin{equation*}  T = \frac{1}{12} \sum_{e \in E} \zeta(e) \varint_{e}^{\star} Df \cdot n_e. \end{equation*}
We have
$f(v_{13}) = 0$, $f(v_{12}) = 0$, $f(v_{23}) = 1$. Moreover
\begin{equation*} w_{1}(v_{23}) = \begin{pmatrix} -1 \\ 0 \end{pmatrix}, \quad w_{2}(v_{23}) = \begin{pmatrix} -1 \\ 1 \end{pmatrix} \end{equation*}
and thus $\mu(v_{23}) = 3/2$. This yields $$S = {3/8}.$$ Now, one can check that 
\begin{equation*} \varint^{\star}_{e_{1}} {\rm D}f \cdot n_{e_{1}}  = -1, \quad \varint^{\star}_{e_{2}}  {\rm D}f \cdot n_{e_{2}}  = 1, \quad \varint^{\star}_{e_{3}}  {\rm D}f \cdot n_{e_{3}}  = 0 \end{equation*}
and 
$\zeta(e_{1}) = -1, \quad \zeta(e_{2}) = -\frac{1}{2}$.
We obtain $T = 1/24$ and therefore $$T_{2}(\Delta,f) = {1/3}.$$ 

Finally, we have $$T_{q}(\Delta,f) = 0,\,\,\,\,\,q \geq 2$$ because the derivatives of $f$ of order greater than $2$ vanish.

\subsection{Numerical computations}

Let us present another example. We choose $\Delta$ as in the previous paragraph, and consider the function
\[ f: \R^2 \setminus \{ (x_1,x_2) \in \R^2; x_1 + x_2 = -1 \} \to \R, \qquad (x_1,x_2) \mapsto \frac{1}{1 + x_1 + x_2}. \]
We want to evaluate the sum
\[ S_N := \frac{1}{N^2} \sum_{k \in N\Delta \cap \mathbb{Z}^2} f\left(\frac{k}{N}\right) = \frac{1}{N^2} \sum_{k_{1}=0}^N \sum_{k_{2} = 0}^{N-k_{1}} \frac{1}{1+\frac{k_1 + k_2}{N}}.\]
Let us compute the coefficients $T_i(\Delta,f)$, $i=0,1,2$, of the asymptotic expansion given in Theorem \ref{thm:main}. Firstly, we have 
\[ T_0(\Delta,f) = \int_0^1\left( \int_0^{1-x_1} \frac{{\rm d}x_2}{1 + x_1 + x_2} \right) {\rm d}x_1 
= \log 2 - \int_0^1 \log(1+x_1) {\rm d}x_1 = 1 - \log 2.\]
Secondly, notice that 
\[ \varint^{\star}_{e_{1}} f = \int_0^1 \frac{{\rm d}x_2}{1 + x_2} = \log 2, \quad \varint^{\star}_{e_{2}} f = \int_0^1 \frac{{\rm d}x_1}{1 + x_1} = \log 2, \quad \varint^{\star}_{e_{1}} f = \frac{1}{2}, \]
and hence 
\[ T_1(\Delta,f) = \frac{1}{4} + \log 2. \]
Finally, since
\[ f(v_{13}) = 1, \quad f(v_{23}) = \frac{1}{2}, \quad f(v_{12}) = \frac{1}{2} \]
and
\[ \mu(v_{13}) = 0, \quad \mu(v_{23}) = \frac{3}{2}, \quad \mu(v_{12}) = \frac{3}{2} \]
we obtain
\[ \sum_{v \in V}\left(\frac{1}{4} {+} \frac{\mu(v)}{12} \right) f(v) = \frac{30}{48} \]
(we deliberately keep this fraction as it is, which will make sense in view of the next result). Moreover, we have
\[ \varint^{\star}_{e_{1}} {\rm D}f \cdot n_{e_{1}} = \int_0^1 \frac{{\rm d}x_2}{(1 + x_2)^2} = \frac{1}{2}, \quad \varint^{\star}_{e_{2}} {\rm D}f \cdot n_{e_{2}} = -\frac{1}{2}, \quad \varint^{\star}_{e_{3}} {\rm D}f \cdot n_{e_{3}} = \frac{1}{2} \]
and
\[ \zeta(e_1) = -1, \quad \zeta(e_2) = -\frac{1}{2}, \quad \zeta(e_3) = -1;\]
thus
\[ \frac{1}{12} \sum_{e \in E} \zeta(e) \varint_{e}^{\star} Df \cdot n_e = \frac{3}{48}. \]
We deduce from the previous results that 
\[ T_2(\Delta,f) = \frac{33}{48}. \]
We can now compute 
\[ P_N := T_0(\Delta,f) +  T_2(\Delta,f) N^{-1} +  T_2(\Delta,f) N^{-2}; \]
in Figure \ref{fig:comparison}, we compare the value of $P_N$ to the numerical value of $S_N$, for $N=10,25,50,75,100,250,500,750,1000$. In Figure \ref{fig:error}, we numerically evaluate the quantity
\[ R_N = \left|S_N - P_N \right|; \]
more precisely, we plot the value of $\log R_N$ as a function of $\log N$. The remainder $R_N$ displays a behavior in $\mathcal{O}(N^{-3})$, as expected.
\begin{figure}
\begin{center}
 \includegraphics[scale=0.5]{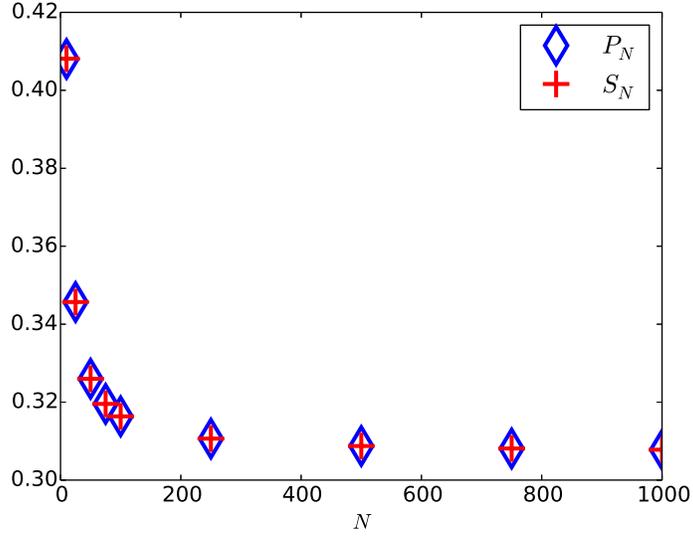}
\end{center}
 \caption{$P_N$ and $S_N$ as functions of $N$.}
 \label{fig:comparison}
\end{figure}
\begin{figure}
\begin{center}
 \includegraphics[scale=0.5]{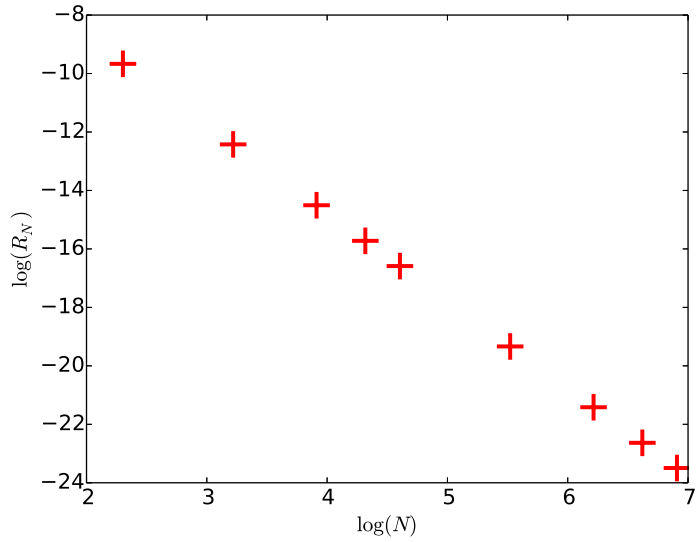}
\end{center}
 \caption{Logarithm of $R_N$ as a function of $\log N$. Notice the behavior in $\mathcal{O}(N^{-3})$ of $R_N$.}
 \label{fig:error}
\end{figure}

\section{Final remarks}

We believe that one should be able to prove item (ii) in Theorem \ref{keyT} in dimensions $n\geq 4$
with the same elementary method we use in this paper (since a similar result was already stated in Tate's article
in any dimension). However, in dimensions $n\geq 4$ the computations appear to be complicated (but
still elementary) and will be the object of future works (see also Remark~\ref{xx}).

\appendix

\section{Comparison between Tate's formula for regular wedges and ours}

We devote this appendix to comparing Tate's formula for Riemann sums over regular wedges \cite[proposition $3.1$]{Ta2010} and the first assertion of theorem \ref{keyT} of the present paper. 

\subsection{Tate's notation and result}

Firstly, we need to recall some of the notation and definitions used in Tate's article \cite{Ta2010}; most of them come from section $3$ of the latter.\\

Let $X$ be a finite dimensional real vector space, of dimension $m = \dim(X)$, and $\Lambda$ be a lattice in $X$. Let $C$ be a unimodular cone (a wedge) in $X$ and $E$ be the integral basis of $\Lambda$ generating $C$. Then $\Z_+^E$ is the set of families of non-negative integers with $\left|E\right|$ elements; an element $\alpha \in \Z_+^E$ is of the form $(\alpha(e))_{e \in E}$. If $e$ belongs to $E$, define $\lambda_e \in \Z_+^E$ in such a way that $\lambda_e(e) = 1$ and $\lambda_e(f) = 0$ if $f \neq e$. For $I \subset E$, let $\Z_+^I$ be the set of $\alpha \in \Z_+^E$ such that for all $e \in E\setminus I$, $\alpha(e) = 0$; then $\alpha \in \Z_+^E$ can be written $\alpha = \sum_{e \in E} \alpha(e) \lambda_e$. If $I \neq \emptyset$ is a subset of $E$, set
\[ \Z_{>0}^I = \left\{ \alpha \in \Z_+^I, \quad \forall e \in I, \alpha(e) \neq 0 \right\};\]
moreover, set $\Z_{>0}^{\emptyset} = \left\{ 0 \right\}$. For $I \subsetneq E$, set 
\[C(I) = \sum_{e \in E \setminus I} \R_+ e,\]
and $C(E) = \left\{ 0 \right\}$. 

For $\nu \in \Z_+^I$ and for $I \subset E$ nonempty, set
\[p_I(\nu) = \prod_{e \in I} p(\nu(e)) \]
with $p(n) = (-1)^n b_n/n!$, where $b_n$ is the $n$-th Bernoulli number. Moreover, for $\mu \in \Z_+^E$, set
\[ \nabla^{\mu} = \prod_{e \in E} \nabla_e^{\mu(e)}.\]
Finally, define the differential operator $L_{n}(C,I)$ as
\begin{equation*} L_n(C,I) = (-1)^n \sum_{\nu \in \Z_{>0}^I, |\nu|=n} p_I(\nu) \nabla^{\nu - e(I)} \end{equation*}
with $e(I) = \sum_{e \in I} \lambda_e$.

The following proposition gives an asymptotic expansion for the Riemann sum 
\begin{equation*} R_N(C,f) = \frac{1}{N^{m}} \sum_{\gamma \in (NC)\cap \Z^m} f(\gamma/N).\end{equation*}
\begin{prop}[Tate {\cite[proposition $3.1$]{Ta2010}}]
Let $f \in \op{C}^{\infty}_0(X)$. Then
\begin{equation} R_N(C,f) \sim \sum_{n \geq 0} N^{-n} \sum_{I \subset E, |I| \leq n} (-1)^{|I|} \int_{C(I)} L_n(C,I)f. \label{formula}\end{equation}
\end{prop}

\newpage

\subsection{Comparison between Tate's notation and ours}

The following table gives an equivalence between Tate's notation and ours.

\begin{center}
\begin{tabular}{|c|c|}
\hline
Tate's notation & Our notation \\
\hline
& \\
$m$ & $n$\\
& \\
\hline 
& \\
$n$ & $q$ \\
& \\
\hline
& \\
$\nu \in \Z_+^E$ & $\alpha \in \N^n$ \\
& \\
\hline
& \\
$E$ & $(-v_1,\ldots,-v_n)$\\
& \\
\hline
& \\
$I \subset E$ (vectors) & $I \subset \llbracket 1,n \rrbracket$ (indices)\\
& \\
\hline
& \\
$C(I)$ & $\bigcap_{i \in I} \mathcal{H}_i$\\
& \\
\hline
& \\
$\Z_{>0}^I$ & $\left\{ \alpha \in \N^n, \quad \left\{ i, \alpha_i \neq 0 \right\} = I \right\}$\\
& \\
\hline
& \\
$e(\left\{ i, \alpha_i \neq 0 \right\})$ & $r(\alpha)$\\
& \\
\hline
& \\
$p_I(\nu)$ & $\frac{(-1)^{|\alpha|}}{\alpha !} \prod_{i=1}^n b_{\alpha_i}$\\
& \\
\hline
& \\
$\nabla^{\mu}f$ & $(-1)^{|\beta|} {\rm D}^{\beta}f \cdot v_{\beta}$ \\
& \\
\hline
& \\
$\int$ & $\varint^{\star}$ \\
& \\
\hline
\end{tabular}
\end{center}

\subsection{Formula (\ref{formula}) with our notation}

Using the previous comparison, $L_n(C,I)f$ becomes with our notation:
\begin{equation*} (-1)^q \sum_{\alpha \in \N^n, \left\{ i, \alpha_i \neq 0 \right\} = I, |\alpha|=q} \frac{(-1)^q}{\alpha !} \left(\prod_{i=1}^n b_{\alpha_i}\right) {\rm D}^{q-\nu(\alpha)}f \cdot v_{\alpha - r(\alpha)}; \end{equation*}
hence, formula (\ref{formula}) becomes
\begin{equation*} R_N(W,f) \sim \sum_{q \geq 0} N^{-q} \sum_{\alpha \in \N^n, |\alpha| = q} \frac{1}{\alpha !} \left(\prod_{i=1}^n b_{\alpha_i}\right) \int^{\star}_{\bigcap_{i, \alpha_i > 0} \mathcal{H}_i} {\rm D}^{q-\nu(\alpha)}f \cdot v_{\alpha - r(\alpha)} \end{equation*}
which coincides with our formula for wedges (first assertion of our main theorem \ref{keyT}).

\vspace{20mm}

\emph{Acknowledgements}. 
AP was supported by NSF CAREER Grant
DMS-1055897 and NSF Grant DMS-0635607. Work on the topic of this paper started when AP was a postdoctoral
fellow at MIT (2007\--2008) and during this period he was supported by an NSF Postdoctoral Fellowship. He
would like to thank Victor Guillemin for many useful discussions on this paper, and on many other topics
related to symplectic geometry and spectral theory, during his stay at MIT. Finally, we thank Mich\`ele Vergne
for comments on a preliminary version of this manuscript.

\vspace{5mm}
\newpage

\noindent
 {\bf  Yohann Le Floch} \\
  School of Mathematical Sciences\\
  Tel Aviv University\\
  Ramat Aviv\\ 
	Tel Aviv 6997801, Israel.\\
  {\em E-mail:} \url{ylefloch@post.tau.ac.il}\\

\noindent
{\bf \'Alvaro Pelayo}\\
University of California, San Diego\\
Department of Mathematics\\
9500 Gilman Dr, {\string #}0112\\
La Jolla, CA 92093\\
United States of America\\
{\em E\--mail}: \url{alpelayo@ucsd.edu}

\end{document}